\documentclass[10pt]{amsart}

\usepackage{wasysym}

\usepackage{bbold}
\usepackage{esint}
\usepackage{enumerate}
\usepackage{amsmath,amsthm,amscd,amssymb}

\usepackage{latexsym}
\usepackage{upref}
\usepackage{verbatim}

\usepackage[mathscr]{eucal}
\usepackage{dsfont}

\usepackage{graphicx}
\usepackage[colorlinks,hyperindex,pagebackref,hypertex]{hyperref}

\usepackage[OT2,OT1]{fontenc}
\newcommand\cyr
{
\renewcommand\rmdefault{wncyr}
\renewcommand\sfdefault{wncyss}
\renewcommand\encodingdefault{OT2}
\normalfont
\selectfont
}
\DeclareTextFontCommand{\textcyr}{\cyr}

\usepackage{pdfpages}

\usepackage{color}
\usepackage{ textcomp }
\usepackage{pstricks}

\usepackage{graphicx}
\usepackage{skull}
\newtheorem{theorem}{Theorem}[section]
\newtheorem{lemma}[theorem]{Lemma}

\theoremstyle{definition}

\newtheorem{hypothesis}[theorem]{Hypothesis}

\theoremstyle{remark}
\newtheorem{remark}[theorem]{Remark}

\numberwithin{equation}{section}

\newcommand{\bbR}{\mathbb R}

\newcommand{\bbC}{\mathbb C}

\renewcommand{\epsilon}{\varepsilon}

\newcommand{\be}{\begin{equation}}
\newcommand{\ee}{\end{equation}}

\newcommand{\Span}{\mathrm{span}}

\newcommand{\R}{\mathbb{R}}

\newcommand{\cB}{{\mathcal B}}

\newcommand{\cH}{{\mathcal H}}

\newcommand{\cP}{{\mathcal P}}

\newcommand{\cU}{{\mathcal U}}

\renewcommand{\Im}{{\ensuremath{\mathrm{Im}}}}

\newcommand{\tr}{\mathrm{tr}}

\newcommand{\fA}{\mathfrak{A}}
\newcommand{\fB}{\mathfrak{B}}
\newcommand{\fC}{\mathfrak{C}}

\DeclareMathOperator{\Ran}{\mathrm{Ran}}
\DeclareMathOperator{\Ker}{\mathrm{Ker}}

\newcommand{\linspan}{\mathrm{lin\ span}}

\newcommand{\Dom}{\mathrm{Dom}}
\newcommand{\dom}{\mathrm{Dom}}

\DeclareFontFamily{U}{rcjhbltx}{}
\DeclareFontShape{U}{rcjhbltx}{m}{n}{<->rcjhbltx}{}
\DeclareSymbolFont{hebrewletters}{U}{rcjhbltx}{m}{n}

\let\aleph\relax\let\beth\relax
\let\gimel\relax\let\daleth\relax

\DeclareMathSymbol{\aleph}{\mathord}{hebrewletters}{39}
\DeclareMathSymbol{\beth}{\mathord}{hebrewletters}{98}
\DeclareMathSymbol{\gimel}{\mathord}{hebrewletters}{103}
\DeclareMathSymbol{\daleth}{\mathord}{hebrewletters}{100}

\DeclareMathSymbol{\lamed}{\mathord}{hebrewletters}{108}
\DeclareMathSymbol{\mem}{\mathord}{hebrewletters}{109}
\DeclareMathSymbol{\ayin}{\mathord}{hebrewletters}{96}
\DeclareMathSymbol{\tsadi}{\mathord}{hebrewletters}{118}
\DeclareMathSymbol{\qof}{\mathord}{hebrewletters}{113}
\DeclareMathSymbol{\shin}{\mathord}{hebrewletters}{152}

\begin{document}

\title [On invariance principle ]
{ On the invariance  principle for a characteristic function  }

\author{K. A. Makarov}
\address{Department of Mathematics, University of Missouri, Columbia, Missouri 63211, USA}
\email{makarovk@missouri.edu}

\author{E. Tsekanovskii }
\address{
 Department of Mathematics, Niagara University, Lewiston,
NY  14109, USA } \email{tsekanov@niagara.edu}

\subjclass[2020]{Primary: 81Q10, Secondary: 35P20, 47N50}

\keywords{ Dissipative operators, characteristic function, deficiency indices,
  quasi-self-adjoint extensions, Krein-von Neumann  extension.\\
  The first author was partially  supported by the Simons collaboration grant 00061759 while preparing this article}

\dedicatory{In respectful memory of Sergei Nikolaevich Naboko}

\begin{abstract} We extend the invariance principle for a characteristic function of a  dissipative operator with respect to the group of affine transformations of the real axis preserving the orientation to the case of general $SL_2(\bbR)$ transformations.
\end{abstract}

\maketitle

\section{Introduction} In \cite{MT-S} we introduced the concept of a characteristic function associated with a triple of operators consisting  of  (i)  a densely defined symmetric operator with deficiency indices $(1,1)$, (ii)  its quasi-self-adjoint dissipative extension    and (iii) a (reference)  self-adjoint extension. This concept turns  out to be convenient in two respects. On the one hand,
 the issue  of choosing an undetermined constant phase factor in the definition of the characteristic function of an unbounded dissipative operator (a  quasi-self-adjoint extension),
which is due to    M.~S.~Liv\v{s}ic \cite{L46} (also see \cite{AkG}), is cleared. On the other hand,  the characteristic function  of a triple determines the whole  triple up to  mutual unitary equivalence, provided that the underlying symmetric operator is prime (see the corresponding uniqueness theorem in \cite{MT-S}). 
 
 From  technical point of view, the characteristic function of a triple has shown  itself as an adequate  tool for solving certain problems in operator theory.  For instance, the solution  of the 
 J$\clock$rgensen-Muhly problem \cite{J80} presented in \cite{MTH,MTBook} (in the particular case of the deficiency indices $(1,1)$)  led to  the complete classification   of simplest  solutions  of the classical commutation relations in the form  
 \begin{equation}\label{1/2Wdot}
U_t \widehat AU_t^*=\widehat  A+t I \quad \text{on } \quad \Dom(\widehat  A), \quad t\in \bbR,
\end{equation}
where  $\widehat A$ is a   dissipative quasi-self-adjoint extension of a symmetric operator $\dot A$ with deficiency indices $(1,1)$  
and  $U_t$ is  a strongly continuous one-parameter group of unitary operators. 
 Recall that the  J$\clock$rgensen-Muhly  problem is to provide 
an intrinsic characterization of  symmetric operators $\dot A$ satisfying the commutation relation \eqref{1/2Wdot}.  In this context  it is worth mentioning  that the solution of the  problem was based on the study of the transformation properties of the characteristic function of a triple of operators (generated by   the  symmetric  and dissipative solutions of \eqref{1/2Wdot}  and augmented by a reference self-adjoint extension of $\dot A$) with respect to affine transformations of the operators of the triple \cite{MTH,MTBook}.

 As opposed to the familiar transformation law  for the characteristic function   $S_{\widehat A}(z)$  of   a bounded dissipative operator $\widehat A$ introduced by M.~S.~Liv\v{s}ic in \cite{Lv1} (also see \cite{LOOW})
 \begin{equation}\label{law}S_{f(\widehat A)}\circ f=S_{\widehat A}\end{equation}
 valid with respect to affine transformations $f(z)=az+b$, $a>0$, $ b\in \bbR$,  
the extension of the  law \eqref{law} to the case of  triples of  unbounded operators  requires appropriate  modifications related to alignment of  relevant phase factors   (see \cite[Theorem F.1, Appendix F]{MTH,MTBook}).

 The goal of this Note is to extend  the transformation law   to the case of general automorphisms  $f\in  Aut (\bbC_+)$ of the upper half-plane $\bbC_+$. The corresponding  
 main result  of this Note  is  as follows, see Theorem \ref{main}.
    
    Given    $f\in Aut (\bbC_+)$ and a  triple $ \fA=(\dot A, \widehat A, A)$,  
  with   $\dot A$  a prime symmetric  operator, $\widehat A$ a maximal dissipative extension of $\dot A$ and $A$ its  (reference) self-adjoint extension, 
we show that if the preimage $\omega=f^{-1}(\infty)$ belongs to the spectrum of  $\widehat A$, then  the  triple
   $$f(\fA)=(f(\dot A),f( \widehat A), f(A))   $$ is well defined. Moreover,   for an appropriately normalized $\widehat S_\fA(z)$  characteristic function of the 
    triple $\fA$ 
 the invariance relation 
   $$
   \widehat S_{f(\fA)}\circ f=\widehat S_{\fA}
   $$
   holds (see \eqref{norm} for the definition of  the normalized  characteristic function  $\widehat S_{\fA}(z)$). 
   
If, instead,  $\omega=f^{-1}(\infty)$ is a regular point of the dissipative operator  $\widehat A$,   we relate the 
  characteristic function $S_{f(\widehat A)}(z)$  of the bounded dissipative operator $f(\widehat A) $ to  the characteristic function $ S_{\fA}(z)$  of the triple as 
  $$
  S_{f(\widehat A)}\circ f=\overline{\Theta_f}S_{\fA}(z),
$$
with  $\Theta_f=S_{\fA}(\omega+i0)$   a  unimodular constant factor.  
 
As an application of the extended invariance principle we obtain
 identities relating the  Friedrichs and Krein-von Neumann extensions of model homogeneous non-negative symmetric  operators 
 and their inverses, see Theorem \ref{last}.

\section{Preliminaries and basic definitions}\label{s2}
Let  $\dot A$ be  a densely defined symmetric operator with deficiency indices $(1,1)$ and   $ A$ its self-adjoint (reference) extension.

Following \cite{D,GT,L46,MT-S}  recall the concept of the Weyl-Titchmarsh
 and   Liv\v{s}ic functions associated with the pair $(\dot A, A)$.

Suppose that (normalized) deficiency elements $g_\pm$,
\begin{equation}\label{start}
g_\pm\in \Ker( \dot A^*\mp iI),\quad 
 \|g_\pm\|=1,
 \end{equation}
  are chosen in such a way that
 \begin{equation}\label{rss}
g_+-g_-\in \dom (A).
\end{equation}

Consider   the {\it Weyl-Titchmarsh function }\footnote{Paying tribute to historical justice it is worth mentioning that the function $M(z)$ has been introduced by Donoghue in \cite{D}.
However, as one  can see from \cite[eq. (5.42)]{GKMT}, it is elementary to express $M(z)$ in terms of the classical Weyl-Titchmarsh function which explains the terminology we use.}
\begin{equation}\label{WTF}
M(z)=
((Az+I)(A-zI)^{-1}g_+,g_+), \quad z\in \bbC_+,
\end{equation} 
associated with the pair $(\dot A, A)$
and   also   the  {\it  Liv\v{s}ic function}
\begin{equation}\label{charsum}
s(z)=\frac{z-i}{z+i}\cdot \frac{(g_z, g_-)}{(g_z, g_+)}, \quad
z\in \bbC_+,
\end{equation}
\begin{equation}\label{start2}
0\ne g_z\in \Ker( \dot A^*- zI), \quad z\in \bbC_+.
\end{equation}

 Recall the important relationship that links  the Weyl-Titchmarsh and  Liv\v{s}ic functions
  \cite {MT-S},
 \begin{equation}\label{blog}
s(z)=\frac{M(z)-i}{M(z)+i},\quad z\in \bbC_+.
\end{equation}

If $\widehat A \ne (\widehat A )^*$ is a  maximal dissipative extension of $\dot A$,
$$
\Im(\widehat A f,f)\ge 0, \quad f\in \dom(\widehat A ),
$$
then $\widehat A$ is automatically quasi-self-adjoint  \cite{ABT, MT-S,Phil,St68} 
and therefore
\begin{equation}\label{parpar}
g_+-\varkappa g_-\in \dom
 (\widehat A  )\quad \text{for some }
|\varkappa|<1.
\end{equation}

 By definition, we call $\varkappa$ the {\it von Neumann parameter}  of the triple  $\fA=(\dot A,  \widehat A,A)$.

Given \eqref{rss} and   \eqref{parpar},  define 
{\it the characteristic function} $S_\fA(z)$ associated with   the triple  $\fA=(\dot A,  \widehat A,A)$  as follows (see \cite{MT-S}, cf. \cite{Lv1})
\begin{equation}\label{ch12}
S_\fA(z)=\frac{s(z)-\varkappa} {\overline{ \varkappa }\,s(z)-1}, \quad z\in \bbC_+,
\end{equation}
where  $s(z)=s_{(\dot A, A)}(z)$ is the Liv\v{s}ic function associated with the pair
$(\dot A, A)$.

We remark that  the von Neumann parameter $\varkappa$ can explicitly be evaluated in terms of the characteristic function of the triple  $( \dot A, \widehat A , A)$ as
\begin{equation}\label{sac}
\varkappa =S_{( \dot A, \widehat A , A)}(i).
\end{equation}
Moreover, as it follows from \eqref{ch12},  the Liv\v{s}ic function associated with the pair
$(\dot A, A)$  admits the representation
\begin{equation}\label{obr}
s(z)=s_{(\dot A, A)}(z)=\frac{S_\fA(z)-\varkappa} {\overline{ \varkappa }\,S_\fA(z)-1}, \quad z\in \bbC_+,
\end{equation}

Summing up, the calculation of the main characteristics (unitary invariants) of the triple $\fA=(\dot A, \widehat A, A)$  can be performed in accordance with the following  algorithm: at the first step, deficiency elements $g_\pm$  that satisfy the relation \eqref{start} are to be found, then one calculates  the 
Weyl-Titchmarsh \eqref{WTF} or/and   Liv\v{s}ic \eqref{charsum} functions
depending of whether  the resolvent of 
the reference self-adjoint operator $A$ or the ``deficiency" field $\bbC_+\ni z\mapsto g_z$ \eqref{start2}
 is available.
In the next step, the von Neumann parameter $\varkappa$  of the triple  \eqref{parpar} is  to be determined and finally, one arrives at the characteristic function $S_\fA(z)$ of  the triple given by \eqref{ch12}.

One can take a different point of view  as
presented in \cite{MT-S}  and in this  way  we come  to     a {\it functional model} of a prime
 dissipative triple in which the characteristic functions is considered as a parameter of the model.

\section{A functional model of a triple}

Given a contractive analytic map $S$,
\begin{equation}\label{chchch}
S(z)=\frac{s(z)-\varkappa} {\overline{ \varkappa }\,s(z)-1}, \quad z\in \bbC_+,
\end{equation}
where $|\varkappa|<1$ and $s(z)$ is
an  analytic, contractive function in $\bbC_+$
satisfying the Liv\v{s}ic criterion \cite{L46} (also see \cite[Theorem 1.2]{MT-S}),
that is,
\begin{equation}\label{vsea0}
s(i)=0\quad \text{and}\quad \lim_{z\to \infty}
z(s(z)-e^{2i\alpha})=\infty \quad \text{for all} \quad  \alpha\in
[0, \pi),
\end{equation}
$$
0< \varepsilon \le \text{arg} (z)\le \pi -\varepsilon,
$$
introduce the function
\begin{equation}\label{s&M}
M(z)=\frac1i\cdot\frac{s(z)+1}{s(z)-1},\quad z\in \bbC_+.
\end{equation}
In this case, the function $M(z)$ admits the representation, 
\begin{equation}\label{murep}
M(z)=\int_\bbR \left
(\frac{1}{\lambda-z}-\frac{\lambda}{1+\lambda^2}\right )
d\mu(\lambda), \quad z\in \bbC_+,
\end{equation}
for some infinite Borel measure $\mu(d\lambda)$,
\begin{equation}\label{infm}
\mu(\bbR)=\infty,
\end{equation}
such that 
\begin{equation}\label{normmu}
\int_\bbR\frac{d\mu(\lambda)}{1+\lambda^2}=1.
\end{equation}

In the Hilbert space $L^2(\bbR;d\mu)$ introduce
 the    (self-adjoint)
operator  $\cB$  of multiplication   by  independent variable 
 on
\begin{equation}\label{nacha1}
\Dom(\cB)=\left \{f\in \,L^2(\bbR;d\mu) \,\bigg | \, \int_\bbR
\lambda^2 | f(\lambda)|^2d \mu(\lambda)<\infty \right \}
\end{equation} 
and denote by  $\dot \cB$  its symmetric
 restriction
on
\begin{equation}\label{nacha2}
\Dom(\dot \cB)=\left \{f\in \Dom(\cB)\, \bigg | \, \int_\bbR
f(\lambda)d \mu(\lambda) =0\right \}.
\end{equation}
Next, introduce
 $\widehat \cB$  as   the dissipative quasi-self-adjoint extension  of the  symmetric operator  $\dot \cB$
 on
\begin{equation}\label{nacha3}
\Dom(\widehat \cB)=\dom (\dot \cB)\dot +\linspan\left
\{\,\frac{1}{\lambda -i}- \varkappa\frac{1}{\lambda +i}\right \},
\end{equation}
where the von Neumann parameter $\varkappa$ of  the triple
$(\dot \cB, \widehat \cB, \cB)$ is  given by 
$$\varkappa=S(i).
$$

Notice that in this case,
\begin{equation}\label{kone}
\Dom(\cB)=
\dom (\dot \cB)\dot +\linspan\left
\{\,\frac{1}{\lambda -i}- \frac{1}{\lambda+i}\right \}.
\end{equation}

We will refer to the triple  $(\dot \cB,   \widehat \cB,\cB)$ as
{\it  the model
 triple } in the Hilbert space $L^2(\bbR;d\mu)$ parameterized by the characteristic functions $S_\fB(z)=S(z)$.
 
\begin{remark}\label{corespec}
 Notice that the core of the spectrum $\widehat \sigma(\dot \cB)$ of the symmetric operator $\dot \cB$ can be characterized as
\begin{equation}\label{corr}\widehat \sigma(\dot \cB)=\left \{s\in \bbR\,\bigg |\,\int_\bbR\frac{d\mu(\lambda)}{(\lambda-s)^{2}}=\infty \quad \text{and } \quad \mu(\{s\})=0\right \}.
\end{equation}
Here  $\mu(d\lambda) $ is the measure \eqref{murep} associated with the pair $(\dot \cB, \cB)$ from the model representation of the triple $\fB$. 
 
Recall that the core  $  \widehat \sigma(\dot \cB) $ of the spectrum of a prime
\footnote{
A symmetric operator
 $\dot A$ is called  a prime operator
if there is no   (non-trivial) subspace invariant under $\dot A$
such that the restriction of $\dot A$ to this subspace is self-adjoint.
}
 symmetric operator $ \dot \cB$ is just the complement of the set of its quasi-regular points,
 $$
 \widehat \sigma(\dot \cB)=\bbR\setminus   \widehat  \rho(\dot \cB).
 $$

 For the convenience of the reader we provide a short proof of \eqref{corr}.
 \begin{proof} As in \cite{SW},
introduce the set 
 $$
 \cP=\left \{s\in \bbR\,\bigg |\,\int_\bbR\frac{d\mu(\lambda)}{(\lambda-s)^{2}}<\infty\quad \text{or } \quad \mu(\{s\})>0\right \},
 $$ 
 which apparently is the  complement of the right hand side of \eqref{corr}. 

We claim that $\cP$ coincides with the set  of quasi-regular points  $\widehat \rho (\dot \cB)$ of the symmetric (prime) operator $\dot \cB$.

Indeed,  it is well known (see, e.g.,  \cite{GM,SW}) that the set $\cP$ does not depend on the choice of the self-adjoint (reference) extension $\cB$ of the  symmetric operator $\dot \cB $.
 Therefore, given $s\in \cP$, without loss we may assume that $\Ker (\cB-sI)=\{0\}$.
 Hence,
 $
 s\in \rho(\cB) \subseteq \widehat \rho(\dot \cB)$, which proves the inclusion
 $$\cP\subseteq \widehat \rho(\dot \cB).
 $$
 On the other hand, if $t\in \widehat \rho(\dot \cB)$, one can always choose a self-adjoint  extension $\cB'$ of $\dot \cB$ such that $t$ is an eigenvalue of $\cB'$. If $d\mu'(d\lambda)$ is the  representing measure associated with the corresponding model representation for  the pair $(\dot \cB, \cB')$
 in the space $L^2(\bbR, d\mu')$, then 
 $$\mu'(\{t\})>0
 $$
 and therefore, $t\in \cP$ (here we have again used  the  independence of the set $\cP$ from the choice of the (reference)  extension $\cB'$).
 Thus,
 $
\widehat \rho(\dot \cB) \subseteq \cP
$
and hence
\begin{equation}\label{P=B}
\cP= \widehat   \rho(\dot \cB),
\end{equation}
 as stated 
and \eqref{corr} follows. 
\end{proof}
\end{remark}

Let  $\dot A$ be  a densely defined symmetric operator with deficiency indices $(1,1)$,  $\widehat A$ a maximal non-selfadjoint dissipative extension of $\dot A$
and 
$ A$ its self-adjoint (reference) extension. Given the triple $\fA=(\dot A, A, \widehat A )$, 
recall that if the symmetric operator $\dot A $ is prime,
then both 
 the Weyl-Titchmarsh function $M(z)$ and  the Liv\v{s}ic function $s(z)$ are
 complete unitary invariants  of the  pair $(\dot A, A)$,  while
the characteristic function $S_\fA(z)$   is a complete unitary invariant of the triple
$\fA$  (see  \cite{MT-S}).
 In particular,
the von Neumann parameter $\varkappa$ is a unitary invariant of the  triple
$(\dot A, \widehat A,A )$, 
not a complete unitary invariant though.
Also notice that if the  symmetric operator  $\dot A$ from a triple 
 in the Hilbert space $\cH$ is not prime and  $\dot A'$ is the  prime part of $ \dot A$
  in a reducing subspace $\cH'\subset \cH$,  then  the triples
$(\dot A, \widehat A, A  )$ and $(\dot A|_{\cH'}, \widehat A|_{\cH'}, A|_{\cH'} )$ have the same characteristic function,
which in many cases allows one  to focus on the case where $\dot A$ is a prime operator.

\begin{theorem}[{\cite[Theorems 1.4, 4.1]{MT-S}}]\label{unitar}

Suppose that $\dot A$ and $\dot B$
 are prime, closed, densely defined  symmetric operators
with deficiency indices $(1,1)$. Assume, in addition, that
  $A$ and $B$ are some self-adjoint
extensions of $\dot A$ and $\dot B$ and that
$\widehat A$ and $\widehat B$ are maximal dissipative
extensions of $\dot A$ and $\dot B$, respectively
$(\widehat A\ne (\widehat A)^*$ and  
$\widehat B\ne (\widehat B)^*)$.

Then,
\begin{itemize}
\item[(i)] the triples  $(\dot A,\widehat A, A)$ and $(\dot
B,\widehat B,  B)$
 are mutually unitarily equivalent\footnote{We say that triples  of operators
$(\dot A, \widehat A, A)$ and $(\dot B,\widehat B,  B)$
 in Hilbert spaces $\cH_A$ and $\cH_B$
are mutually unitarily equivalent if there is
a unitary map $\cU$ from $\cH_A$ onto $\cH_B$ such that
$\dot B=\cU\dot A\cU^{-1}$, $\widehat  B=\cU\widehat  A\cU^{-1}$, and  $ B=\cU A\cU^{-1}$.}
if, and only if, the corresponding characteristic functions   of the triples coincide;

\item[(ii)] the triple  $(\dot A,\widehat A,  A)$ is mutually unitarily
equivalent to the model triple 
$(\dot \cB, \widehat\cB,  \cB )$
in
the Hilbert space $L^2(\bbR;d\mu)$, where $\mu(d\lambda)$ is the
representing measure for the Weyl-Titchmarsh function $M(z)=M_{(\dot A,
A)}(z)$ associated with the pair $(\dot A, A)$.
\end{itemize}

In particular,
\begin{itemize}\item[(iii)] the pairs $(\dot A, A)$ and $ (\dot B, B)$ are mutually unitary equivalent if and only if $M_{(\dot A, A)}(z)=M_{(\dot B, B)}(z)$.
\end{itemize}

\end{theorem}

For the further reference  recall the following resolvent formula \cite{MT-S}.
\begin{theorem}\label{modeldef} Suppose that
$\fB=(\dot \cB, \widehat \cB, \cB)$ is the model triple in the Hilbert space
$L^2(\bbR;d\mu) $ given by \eqref{nacha1}-\eqref{nacha3}.

 Then the
resolvent
 of the model dissipative operator $\widehat \cB$  in
$L^2(\bbR;d\mu) $ has the form
\begin{equation}\label{rezfor1}
(\widehat \cB- zI )^{-1}=(\cB- zI )^{-1}-p(z)(\cdot\, ,
g_{\overline{z}})g_z ,
\end{equation}
 with
\begin{equation}\label{rezfor12}
p(z)=\left (M_{(\dot \cB,
\cB)}(z)+i\frac{\varkappa+1}{\varkappa-1}\right )^{-1},
\end{equation}
$$z\in\rho(\widehat \cB)\cap \rho(\cB).
$$
Here $M_{(\dot \cB,
\cB)}(z)$ is the Weyl-Titchmarsh function associated with the pair
 $(\dot \cB, \cB)$ continued to the lower half-plane by the Schwarz reflection
principle, $\varkappa$ is the von Neumann parameter of the  triple $\fB$,
and the deficiency elements $g_z$,
$$
g_z\in \Ker ((\dot \cB)^*-zI), \quad z\in \bbC\setminus \bbR,
$$
are given by
\begin{equation}\label{defelrep}
g_z(\lambda)=\frac{1}{\lambda-z} \quad
\text{for $\mu$-almost all  } \lambda\in \bbR.
\end{equation}

\end{theorem}


\begin{remark}\label{obrat1}  
Notice that if   $z=0$
is a quasi-regular point for the symmetric operator $\dot \cB$, $0\in \widehat \rho(\dot \cB)$,  and therefore
$$
0\in\rho(\cB)\cap\rho(\widehat \cB), 
$$
then the inverse $\widehat \cB^{-1}$ is a rank-one perturbation of
 the bounded self-adjoint operator $\cB^{-1}$ 
and the following resolvent formula
$$
\widehat \cB^{-1}=\cB^{-1}-p Q
$$
holds.
Here
\begin{equation}\label{pppp}
p=\left (M(0)+
i\frac{\varkappa+1}{\varkappa-1}\right)^{-1},
\end{equation}
$Q$ is a rank-one self-adjoint operator given by
$$
(Qf)(\lambda)=\frac{1}{\lambda}
\int_\bbR \frac{f(s)}{s}d\mu(s), \quad
\text{ $\mu$-a.e. } \lambda \in \bbR,
$$
and $M(0)=M(0+i0)$ is the boundary value of the Weyl-Titchmarsh function associated with the pair $(\dot \cB, \cB)$
at the point zero.

\end{remark}


\section{The invariance principle }

Let  $\dot A$ be  a densely defined symmetric operator with deficiency indices $(1,1)$,   $ A$ its self-adjoint (reference) extension and
$\widehat A$ a maximal non-self-adjoint dissipative extension of $\dot A$.
In this case, the dissipative operator $\widehat A$ is automatically a quasi-self-adjoint extension of $\dot A$ (see, e.g., \cite{MT-S}),
 $$
\dot A\subset \widehat A\subset  (\dot A)^*
$$
and 
$$\dot A= \widehat A |_{\Dom (\widehat A)\cap \Dom(\widehat A^*)}.$$ 
Throughout this Note such kind of triples $\fA=(\dot A, \widehat A, A)$ will be called {\it regular}.

First, 
recall the concept of  an invariance principle  for the characteristic function of a triple
 $\fA=(\dot A,  \widehat A,A)$
 with respect to   affine
transformations of the 
  operators $\dot A$,  $\widehat A$ and $A$ from $\fA$.

Given an affine transformation of the upper half-plane,
$
f(z)=az+b
$,
 $a, b\in \bbR$, $a>0$,  introduce the triple
 $$
 f(\fA)=(f(\dot A),f ( \widehat A),f(A)),
 $$
where
$$
f(X)=aX+bI \quad \text{on}\quad  \Dom (X),\quad X=\dot A, \widehat A, A.
$$

 Clearly, $f(\dot A)$ is a symmetric operator with deficiency indices $(1,1)$,
 $f(A)$ is its self-adjoint extension, and $f(\widehat A)$ is a quasi-self-adjoint dissipative extension of $f(\dot A)$. Therefore, the triple   $f(\widehat A)$ is regular and hence 
 the characteristic function $S_\fA(z)$  of the triple $f(\fA)$ is well defined.
 
 Along with the characteristic function $S_\fA(z)$   introduce the normalized characteristic function  $\widehat S_\fA(z)$  of the triple $\fA$ as 
\begin{equation}\label{norm}
 \widehat S_\fA(z)=\frac{1-\overline{S_\fA(i)}}{1-S_\fA(i)}\cdot S_\fA(z), \quad z\in \bbC_+.
 \end{equation}

\begin{theorem}[{\cite[Theorem F.1]{MTH,MTBook}}]\label{coin}
Let $\fA=(\dot A, \widehat A, A)$ be a regular triple.
Suppose that
$
f(z)=az+b
$
with $a, b\in \bbR$, $a>0$, is an affine transformation.

Then for the  normalized characteristic functions of the triples $\fA$ and $f(\fA)$ the  following invariance principle 
\begin{equation}\label{inveq}
 \widehat S_{f(\fA)}\circ f= \widehat  S_{\fA}
\end{equation}
holds.

\end{theorem}

If $\fA$ is a regular triple and  
 $f$ is a linear-fractional automorphism of the upper half-plane  $\bbC_+$, 
$f\in SL_2(\bbR)$,
 one can also introduce the triple 
$$
f(\fA)=(f(\dot A), f(\widehat A),f (A)),
$$
where the function $f(X)$ of a linear operator $X$ is understood as
$$f(X)=(aX+cI)(cX+dI)^{-1} \quad \text{on} \quad  \Ran \left ((cX+dI)\right ), \quad X=\dot A, \widehat A, A,
$$
provided that $\Ker(cX+dI)=\{0\}$.
However, if the preimage $f^{-1}(\infty)$ is a quasi-regular point for the symmetric operator $A$, then the operator $f(\dot A)$ is not densely defined and therefore the triple $f(\fA)$ is not regular.

In fact, if   $\dot A$ is a prime  symmetric operator and
 $f\in SL_2(\bbR)$ is
a linear-fractional automorphism, then we have the following alternative: either the triple  
$$
f(\fA)=(f(\dot A), f(\widehat A),f (A))
$$ is regular or $f(\widehat A)$ is a bounded dissipative operator. 

The following main result of the Note  takes care of the two possible outcomes
in the alternative  mentioned above.

\begin{theorem}\label{main} 
Suppose that   $\fA=(\dot A, \widehat A, A)$ is a regular triple.  Assume, in addition, that   $\dot A$ is a prime  symmetric operator.
Suppose that  $f\in SL_2(\bbR)$ is a linear-fractional automorphism of the upper half-plane. 

Then, 
\begin{itemize}
\item[(i)]
if $\omega=f^{-1}(\infty) $
belongs  to  the spectrum of  the dissipative operator $\widehat  A$, equivalently, to the core of the spectrum of  the symmetric operator $\dot A$, then the triple 
$$
f(\fA)=(f(\dot A), f(\widehat A),f (A))
$$
is regular. 

In this case,
 for the  normalized characteristic functions  $\widehat S_{\fA}(z)$ and $\widehat S_{f(\fA)}(z)$ 
 of the triples $\fA$ and $f(\fA)$ the  following invariance principle 
\begin{equation}\label{inveqq}
 \widehat S_{f(\fA)}\circ f=   \widehat  S_{\fA}
 \end{equation}
holds;
\item[(ii)]
if   $\omega=f^{-1}(\infty) $ is a regular point  of the dissipative operator  $\widehat  A$, equivalently, 
 a quasi-regular point  of  the symmetric operator $\dot A$, then the operator $f(\widehat A)$  is well defined as a bounded dissipative operator.
 
In this case, for the  normalized characteristic function   $\widehat S_{\fA}(z) $ of  the triple $\fA$ and the characteristic function  $S_{f(\widehat A)}(z)$ of the bounded dissipative operator $f(\widehat A)$ the  following invariance principle 

\begin{equation}\label{inveq2}
S_{f(\widehat A)}\circ f=\frac{1}{S_{\fA}\left (\omega+i0\right )}\cdot S_{\fA}
\end{equation}
holds.

\end{itemize}

\end{theorem}

\begin{remark} Notice that the technical requirement that $\dot A$ is a prime operator can easily be relaxed with obvious modifications in the formulation.
\end{remark}

We defer  the proof of Theorem \ref{main} to Section \ref{Secproof}.

\section{Invariance principe for Model Triples}

Taking into account that  each automorphism of the upper half-plane  is either a 
linear transformation 
of $\bbC_+$ or 
 a composition of   linear transformations of $\bbC_+$ and the automorphism
\begin{equation}\label{1z}
f(z) =  -\frac{1}{z},\quad z \in \bbC_+,
 \end{equation}
we will  concentrate first  on the case where the relevant  mapping  is given by  
 \eqref{1z}. 
Under this assumption we will establish the corresponding invariance principle for the model triple of operators $\fB=(\dot \cB, \widehat \cB, \cB)$
 given by   \eqref{nacha1}-\eqref{nacha3}.

Notice, that in the  situation in question,  the symmetric operator $\dot \cB$ is  automatically a prime operator and hence 
  we have the following alternative: either the  point $z=0$  belongs to the core of the spectrum of $\dot \cB$ or zero  is a quaisi-regular point for that operator (see \cite{AkG}).

We start with the case where the point $z=0$ belongs to the core of the spectrum of the model symmetric operator $\dot \cB$.

\begin{theorem}\label{takitak} 
Suppose that
$\fB=(\dot \cB, \widehat \cB, \cB)$ is the model triple  in $L^2(\bbR;d\mu)$ given by \eqref{nacha1}-\eqref{nacha3}. 
Assume that  $0$ belongs to the core of the spectrum $\widehat \sigma(\dot \cB)$ of the symmetric operator $\dot \cB$.
 
Then,

\begin{itemize} 
\item[(i)] $(\dot \cB)^{-1}$ is a symmetric operator with deficiency indices $(1,1)$;
\item[(ii)]$\cB^{-1}$is a  self-adjoint extension  of  $(\dot \cB)^{-1}$;
\item[(iii)]$(\widehat \cB)^{-1}$ is a quasi-self-adjoint extension of  $(\dot \cB)^{-1}$;

In particular,  the triple 
\begin{equation}\label{triplong}\fC=(-(\dot \cB)^{-1}, -(\widehat \cB)^{-1},- \cB^{-1})
\end{equation}
is well defined and regular.

Moreover, 
\end{itemize}
\begin{itemize}
\item[(iv)]the von Neumann parameters of the triples $\fB$ and $\fC$ coincide;
\item[(v)]
  the Liv\v{s}ic functions  $s_{ (-(\dot \cB)^{-1}, - \cB^{-1})}(z) $ and $s_{(\dot \cB, \cB)}(z)$ associated with the pairs $ (-(\dot \cB)^{-1}, - \cB^{-1})
$ and  $(\dot \cB, \cB)$  are  related as 
 \begin{equation}\label{s101}
s_{ (-(\dot \cB)^{-1}, - \cB^{-1})
}\left (-\frac1z\right )=s_{(\dot \cB, \cB)}(z)
,\quad z\in \bbC_+;
\end{equation}
\item[(vi)]
 the Weyl-Titchmarsh  functions  $M_{ (-(\dot \cB)^{-1}, - \cB^{-1})}(z) $ and $M_{(\dot \cB, \cB)}(z)$ associated with the pairs $ (-(\dot \cB)^{-1}, - \cB^{-1})
$ and  $(\dot \cB, \cB)$  are  related as 
 \begin{equation}\label{m101}
M_{ (-(\dot \cB)^{-1}, - \cB^{-1})
}\left (-\frac1z\right )=M_{(\dot \cB, \cB)}(z)
,\quad z\in \bbC_+;
\end{equation}

\item[(vii)] the characteristic functions $S_\fC(z)
$ and   $S_\fB(z)$ of the triples $$\fC=(-(\dot \cB)^{-1}, -(\widehat \cB)^{-1},- \cB^{-1})\quad \text{and}\quad 
\fB=(\dot \cB, \widehat \cB, \cB)$$ satisfy 
the  invariance identity
\begin{equation}\label{nui}
S_{\fC}\left (-\frac1z\right )=S_\fB(z)
,\quad z\in \bbC_+.
\end{equation}
\end{itemize}
\end{theorem}

\begin{proof} Since $\dot \cB$ is a prime symmetric operator and the point $z=0$ is not a quasi-regular point of $\dot \cB$, that is 
$0\in \widehat \sigma(\dot \cB)$,
the subspace $\Ker((\dot \cB)^*)$ is trivial. For the convenience of the reader, we present the corresponding argument.

Indeed,  suppose on the contrary  that  $\Ker((\dot \cB)^*)\ne\{0\}$. Then the restriction $\cB'$ of $(\dot \cB)^*$ on  
$$\dom(\cB')=\dom(\dot \cB)\dot +\Ker((\dot \cB)^*)$$
is a self-adjoint operator. Let $\mu'(d\lambda)$ be the measure from the representation
for the Weyl-Titshmarsh function $M_{(\dot \cB, \cB')}(z)$ associated with the pair $(\dot \cB, \cB')$, 
$$
M_{(\dot \cB, \cB')}(z)=\int_\bbR \left
(\frac{1}{\lambda-z}-\frac{\lambda}{1+\lambda^2}\right )
d\mu'(\lambda), \quad z\in \bbC_+.
$$
Taking into account that $\Ker(\cB')=\Ker((\dot \cB)^*)\ne\{0\}$, we see that  zero is an eigenvalue of $\cB'$ and hence
$$
\mu'(\{0\})\ne 0.
$$
However, 
since the right hand side of \eqref{corr} after replacing $\mu(d\lambda)$ with $\mu'(d\lambda)$ remains invariant
(see, e.g., \cite{GM}, cf.  Remark \ref{corespec}),  the membership $0\in \widehat \sigma(\dot \cB)$ implies
$$
\mu(\{0\})=\mu'(\{0\})= 0.
$$
 The obtained  contradiction shows that  
\begin{equation}\label{kerker}
\Ker((\dot \cB)^*)=\{0\}.
\end{equation}
 
From \eqref{kerker} it follows that $\Ran(\dot \cB)$ is dense in $L^2(\bbR;d\mu)$. In particular,  $\dot B=(\dot \cB)^{-1}$ is well defined as a densely-defined unbounded symmetric operator. Also,
$$
\Ker (\cB)\subseteq\Ker((\dot \cB)^*)=\{0\}
$$
and therefore the inverse $B=\cB^{-1}$ of the self-adjoint (multiplication) operator $\cB$ is well defined as a self-adjoint operator.
Clearly,   $\dot B$ coincides with   the restriction of the unbounded self-adjoint operator
$
B=\cB^{-1}
$
on 
\begin{equation}\label{oblast*1}
\dom(\dot B)=\left \{f \in \Dom (B)\, \bigg |\, \int_\bbR
\frac{f(\lambda)}{\lambda} d\mu(\lambda)=0\right \}.
\end{equation}

Using \eqref{oblast*1}, we see that 
 for any $f\in \dom(\dot B)=\dom((\dot \cB)^{-1})\subset \dom(( \cB)^{-1})$ we have
\begin{align*}
(\dot B-\overline{z}I)f, h_z)&=
 ((\cB^{-1}-\overline{z}I)f, h_z)=
\int_\bbR \left (\frac{1}{\lambda}-\overline{z}\right )
f(\lambda)\overline{\frac{1}{1-\lambda z}}d\mu(\lambda)
\\&=\int_\bbR \frac{f(\lambda)}{\lambda}d\mu(\lambda)
=0.
\end{align*}
Therefore, the functions 
\begin{equation}\label{defele11}
h_z(\lambda)=\frac{1}{1-\lambda z},\quad \Im(z)\ne 0,
\end{equation}
 are  deficiency elements of $\dot B$. That is, 
\begin{equation}\label{defele21}
h_z\in \Ker ((\dot B)^*-zI), \quad \Im(z)\ne 0.
\end{equation}
In particular, 
 the symmetric operator $\dot B$ has deficiency indices $(m,n)$ with $m,n\ge 1$. 

To show that $m=n=1$, observe that for any $f\in\dom (B)= \Dom(\cB^{-1})$ the function 
\begin{equation}\label{fhat}
\widehat f(\lambda)=f(\lambda)-\int_\bbR\frac{f(s)}{s}d\mu(s)\cdot \frac{\lambda}{\lambda^2+1}\end{equation}
has the property that 
$$
\int_\bbR\frac{\widehat f(\lambda)}{\lambda}d\mu(\lambda)=\int_\bbR\frac{f(\lambda)}{\lambda}d\mu(\lambda)-\int_\bbR\frac{f(s)}{s}d\mu(s)
\int_\bbR\frac{d\mu(\lambda)}{\lambda^2+1}=0,
$$
so that 
\begin{equation}\label{naser}
\widehat f\in  \dom(\dot B).
\end{equation}

Here we have used the normalization condition \eqref{normmu}.

From \eqref{defele11} and \eqref{defele21}  it follows that the function 
$$r(\lambda)=\frac{\lambda}{\lambda^2+1}$$
 belongs to the subspace
$ \Ker (( \dot B )^*-iI) \dot +\Ker ((\dot B )^*+iI)$. In particular,  \eqref{fhat} along with \eqref{naser}
show that  the quotient space $ \dom(B)/\dom (\dot B)$ is one-dimensional  and hence
$$m=n=1.
$$
Thus, $\dot B$ is a symmetric restriction  with deficiency indices $(1,1)$ of the self-adjoint operator $B$
and the proof of (i) and (ii)  is complete.

To check (iii), first observe that $\widehat B=(\widehat \cB)^{-1}$ is well-defined, since
$$
\Ker(\widehat \cB)\subseteq\Ker ((\dot \cB)^*)=\{0\}.
$$
Next we claim  that the operator $\widehat B=(\widehat \cB)^{-1}$ is the quasi-self-adjoint extension of $\dot B$
 on 
\begin{equation}\label{oblast11}
\Dom ( \widehat B)=\dom (\dot B) + \Span
\left \{\frac{1}{\lambda -i}+\varkappa \frac{1}{\lambda +i}\right \},
\end{equation}
where $\varkappa$ is the von Neumann parameter of the model  triple $\fB=(\dot \cB, \widehat  \cB, \cB)$. 

Indeed, since (see \eqref{nacha3})
\begin{equation}\label{oblastobr}
\Dom ( \widehat \cB)=\dom (\dot \cB) + \Span
\left \{\frac{1}{\lambda -i}-\varkappa \frac{1}{\lambda +i}\right \}
\end{equation}
and $\dot B=(\dot \cB)^{- 1}$, it suffices to check the following two (mapping) properties for the operators 
$\widehat \cB$ and $\widehat B$,
\begin{align*}
\widehat \cB \left (\frac{1}{\lambda -i}-\varkappa \frac{1}{\lambda +i}\right )=
(\dot  \cB)^* \left (\frac{1}{\lambda -i}-\varkappa \frac{1}{\lambda +i}\right )
= \frac{i}{\lambda -i}+\varkappa \frac{i}{\lambda +i}
\end{align*}
and 
\begin{align*}
\widehat B   \left (\frac{i}{\lambda -i}+\varkappa \frac{i}{\lambda +i}\right )&=
(\dot  B)^*   \left (\frac{i}{\lambda -i}+\varkappa \frac{i}{\lambda +i}\right )=
i\left (\frac{(-i)}{\lambda -i}+\varkappa \frac{i}{\lambda +i}\right )
\\&=
\frac{1}{\lambda -i}-\varkappa \frac{1}{\lambda +i}.
\end{align*}
Here we used that by  \eqref{defelrep},
\begin{equation}\label{defka1}
\Ker ((\dot \cB)^*\mp iI)=\Span \left \{h_\pm\right \},
\end{equation}
and also that
(see \eqref{defele11} and  \eqref{defele21})
\begin{equation}\label{defka}
\Ker ((\dot B)^*\mp iI)=\Span \left \{h_\mp \right \},
\end{equation}
with
$$
h_\pm (\lambda)=\frac{1}{\lambda\mp i}.
$$
The proof of (iii) is complete.

Finally, to calculate the von Neumann parameter of the triple $\fC$ given by \eqref{triplong} we proceed as follows.

Observing that
$$
h_+(\lambda)+ h_-(\lambda)=\frac{1}{\lambda-i}+\frac{1}{\lambda+i}=2\frac{\lambda}{\lambda^2+1},
$$
we see that 
\begin{equation}\label{m1}
h_++h_-\in \dom (B)=\dom (-B).
\end{equation}
Moreover, from \eqref{oblast11} it follows that 
\begin{equation}\label{m2}
h_++\varkappa h_-\in \dom (\widehat B)=\dom (-\widehat B).
\end{equation}

Since
$$
\Ker (-(\dot B)^*\mp iI)=\Ker ((\dot B)^*\pm iI)=\Span\{h_\pm\},
$$
and by \eqref{defka}
$$
\Ker ((\dot B)^*\pm iI)=\Span\{h_\pm\},
$$
the memberships \eqref{m1} and \eqref{m2} ensure that $\varkappa$ is the von Neumann parameter of the triple 
$$\fC=(\dot B, \widehat B, B)=(-(\dot \cB)^{-1}, -(\widehat \cB)^{-1},- \cB^{-1}),$$
which completes the proof of (iv).

Next, we will evaluate  the Liv\v{s}ic functions associated with the pairs
 $(\dot \cB, \cB)$ and $(-\dot B, -B)$. 

Recall that by   \eqref{defelrep},
$$ \Ker ((\dot \cB)^*-z I)=\Span\left \{g_z \right \}, \quad \Im (z) \ne 0,  
$$
where 
$$g_z(\lambda)=\frac{1}{\lambda-z}, \quad \Im (z) \ne 0.
$$
Since
$$
h_+-h_-\in \Dom (\cB),\quad h_\pm \in\Ker ((\dot \cB)^*\mp iI),
$$
for the Liv\v{s}ic function $s_{(\dot \cB, \cB)}(z)$ associated with the pair  $(\dot \cB, \cB)$
we obtain  the representation
\begin{equation}\label{s121}s_{(\dot \cB, \cB)}(z)=\frac{z-i}{z+i}\cdot \frac{(g_z,h_+)}{(g_z,h_-)}=
\frac{z-i}{z+i}\cdot \frac{\int_\R \frac{d\mu(\lambda)}{(\lambda-z)(\lambda-i)}}
{\int_\R \frac{d\mu(\lambda)}{(\lambda-z)(\lambda+i)}}.
\end{equation}

From   \eqref{defele11} and   \eqref{defele21} it follows that
$$ \Ker (-(\dot B)^*-z I)=\Span\left \{\hat h_z \right \}, \quad \Im (z) \ne 0,  
$$
where 
$$\hat h_z(\lambda)=\frac{1}{1+\lambda z}, \quad \Im (z) \ne 0.
$$
Observing that 
$$
h_++h_-=h_+-(-1)h_-\in \Dom (B),\quad h_\pm \in\Ker (-(\dot B)^*\mp iI),
$$
in accordance with the definition of  the Liv\v{s}ic function $s_{(-\dot B, -B)}(z)$ associated with the pair  $(-\dot B, -B)$
we obtain 
\begin{equation}\label{s122}s_{(-\dot B,- B)}(z)=\frac{z-i}{z+i}\cdot \frac{(\hat h_z,h_+)}{(\hat h_z,(-1)h_-)}=
-\frac{z-i}{z+i}\cdot \frac{\int_\R \frac{d\mu(\lambda)}{(1+\lambda z)(\lambda-i)}}
{\int_\R \frac{d\mu(\lambda)}{(1+\lambda z)(\lambda+i)}}, \quad z\in \bbC_+.
\end{equation}

A simple computation using \eqref{s121} and \eqref{s122} shows that
\begin{align*}
 s_{(-\dot B, -B)}\left (-\frac1z\right)
=
s_{(\dot \cB, \cB)}(z),\quad z\in \bbC_+,
\end{align*}
thus proving \eqref{s101}. The proof of (v) is complete.

The assertion (vi) is a direct consequence of (v) and the relationship \eqref{s&M} linking the Liv\v{s}ic  and Weyl-Titchmarsh functions.

The last assertion (vii) is a consequence of (iv)  and (v) and the definition of the characteristic function of a triple.

\end{proof}

\begin{remark}\label{ssylka} Since  the point $z_0=i$ is a fixed point of the automorphism $$f(z)=-\frac1z,\quad z\in \bbC_+,$$ 
using 
 \eqref{nui}  we see that 
$$
S_{\fC}(i)=S_{\fC}\left (-\frac1i\right )=S_{\fB}(i).
$$
Therefore, for the normalized characteristic functions 
$$\widehat S_{f(\fB)}(z)=
\widehat S_{\fC}(z)=\frac{1-\overline{S_{\fC}(i)}}{1-S_{\fC}(i)}\cdot S_{\fC}(z),\quad z\in \bbC_+,
$$ 
and  
$$\widehat S_{\fB}(z)=
\frac{1-\overline{S_{\fB}(i)}}{1-S_{\fB}(i)}\cdot S_{\fB}(z), \quad z\in \bbC_+,
$$
associated with the triples 
$$f(\fB)=\fC=
\left (-(\dot \cB)^{-1}, -(\widehat \cB)^{-1},- \cB^{-1}\right ) \quad \text{and}
\quad\fB=(\dot \cB, \widehat \cB, \cB),$$ respectively,  we also have the invariance equality
$$
\widehat S_{f(\fB)}\circ f=\widehat  S_\fB.
$$
\end{remark}

Notice that if the point $z=0$ is a quasi-regular point of the symmetric operator $\dot \cB$, then the operator $-(\dot \cB)^{-1}$ is well defined as a bounded dissipative operator,  while $\Ker (\dot \cB)^*$ is non-trivial and  therefore $(\dot B)^{-1}$ is not densely defined (although $(\dot \cB)^{-1}$ is continuous on $\Ran (\dot \cB)$).
 Therefore, the triple 
$$ f(\fB)=
\left (-(\dot \cB)^{-1}, -(\widehat \cB)^{-1},- \cB^{-1}\right ) $$  is not regular in this case.
However, we have
the following result (cf.  \cite[Theorem 8.4.4]{ABT}, \cite{TER} in system theory that relates the transfer functions of $L$-systems under  the transformation $z\mapsto \frac1z$ of the spectral parameter). 

\begin{lemma}[{cf. \cite[Theorem 8.4.4]{ABT},\cite{TER}}]\label{tak}

Suppose that
$\fB=(\dot \cB, \widehat \cB, \cB)$ is the model triple  in $L^2(\bbR;d\mu)$ given by \eqref{nacha1}-\eqref{nacha3}. Assume that the point 
 $z=0$ is a quasi-regular point of the symmetric operator $\dot \cB$.

In this case, the operator $\widehat \cB$ has a bounded inverse
and  the characteristic function
$S_{-(\widehat \cB)^{-1}}(z)$ 
 of the dissipative operator $-(\widehat \cB)$ and the characteristic function  $S_{\fB}(z)$ of the triple $\fB$ are related as 
\begin{equation}\label{kep}
 S_{-(\widehat \cB)^{-1}}\left (-\frac1z\right )=\frac{S_{\fB}(z)}{
S_{\fB} (0+i0)},  \quad z\in \bbC_+.
\end{equation}

\end{lemma}

\begin{proof}   Assume temporarily  that  $ \Ker( \cB)=\{0\}.$ Then 
under the hypothesis that $0$ is a quasi-regular point of the symmetric operator $\dot \cB$, we have  
$$0\in \rho(\cB)\cap \rho(\widehat \cB).
$$
By Corollary \ref{obrat1}, the inverse $ \widehat \cB^{-1}$ is a rank-one perturbation of   the self-adjoint operator $\cB^{-1}$,
\begin{equation}\label{resformula1}
\widehat \cB^{-1}=\cB^{-1}-p\, Q.
\end{equation}
Here
\begin{equation}\label{pp}
p=\left (M(0)+
i\frac{\varkappa+1}{\varkappa-1}\right)^{-1},
\end{equation}
 $M(0)=M(0+i0)$  is  the value of the Weyl-Titchmarsh function $M(z)=M_{(\dot \cB, \cB)}(z)$ associated with the pair $(\dot \cB, \cB)$
at the point zero, $\varkappa$  is  the von Neumann parameter of the triple   $\fB=(\dot \cB, \widehat  \cB, \cB)$,  and 
$Q$ is a rank-one self-adjoint operator given by 
$$
(Qh)(\lambda)=\frac{1}{\lambda}
\int_\bbR \frac{h(\lambda)}{\lambda}d\mu(\lambda), \quad
\text{ $\mu$-a.e. } \lambda.
$$

In accordance with the definition \cite{Lv1}
 of the characteristic
 function $S_{-(\widehat \cB)^{-1}}(z)$ of the bounded dissipative operator
$-\widehat \cB^{-1}$ we have
\begin{align}
S_{-(\widehat \cB)^{-1}}(z)&=1+2i \,\Im(p)\,\tr
\left [\left (-(\widehat\cB^{-1})^*-zI\right )^{-1}Q\right ]
\nonumber \\&
\nonumber\\&=1+2i\,  \Im(p)\,\tr
\left [\left (-\cB^{-1}+\overline{p}Q-zI\right )^{-1}Q\right ],
\nonumber
\end{align}
where we have used \eqref{resformula1} in the last step.  Since $ \overline{p}Q $
is a   rank-one perturbation  of the self-adjoint operator $-\cB^{-1}$,  from the first resolvent identity it follows that (see, e.g., \cite{Simon})
$$
\tr
\left [\left (-\cB^{-1}+\overline{p}Q-zI\right )^{-1}Q\right ]=
\frac{\tr
\left [\left (-\cB^{-1}-zI\right )^{-1}Q\right ]}{1+\overline{p}\, \tr
\left [\left (-\cB^{-1}-zI\right )^{-1}Q\right ]},
$$
 and hence
\begin{equation}\label{charb}
S_{-(\widehat \cB)^{-1}}(z)=\frac{1+p\, \tr
\left [\left (-\cB^{-1}-zI\right )^{-1}Q\right ]}{1+\overline{p}\, \tr
\left [\left (-\cB^{-1}-zI\right )^{-1}Q\right ]}
=
\frac{1-p\int_\bbR\frac{d\mu(\lambda)}{\lambda(1+\lambda z)}}
{1-\overline{p}\int_\bbR\frac{d\mu(\lambda)}{\lambda(1+\lambda z)}}
,\quad z\in \bbC_+.\end{equation}

On the other hand, using  \eqref{blog} and  \eqref{ch12}, for the characteristic function $ S_{\fB} (z )$  of the triple  $\fB$ one obtains
$$
S_{\fB} (z )=-\frac{1-\varkappa}{1-\overline{\varkappa}}\cdot
\frac{M\left (z\right )-i\frac{1+\varkappa}{1-\varkappa}}{M\left (z\right )+i\frac{1+\overline{\varkappa}}
{1-\overline{\varkappa}}}, z\in \bbC_+,
$$
and therefore
$$
S_{\fB}\left (-\frac1z\right )=-\frac{1-\varkappa}{1-\overline{\varkappa}}\cdot
\frac{M\left (-\frac1z\right )-i\frac{1+\varkappa}{1-\varkappa}}{M\left (-\frac1z\right )+i\frac{1+\overline{\varkappa}}
{1-\overline{\varkappa}}}, \quad z\in \bbC_+.
$$
From \eqref{pp} it also follows that
$$
p^{-1}-M(0)=
-i\frac{\varkappa+1}{1-\varkappa}
$$
and hence
\begin{align*}
S_{\fB}\left (-\frac1z\right )&
=-\frac{1-\varkappa}{1-\overline{\varkappa}}\cdot
\frac{M\left (-\frac1z\right )-M(0)+p^{-1}}
{M\left (-\frac1z\right )-M(0)+\overline{p^{-1}}}, \quad z\in \bbC_+.
\end{align*}
Here we used that $M(0)$ is real for $0\in \rho(\cB)$.

Since 
\begin{align*}
\frac{M\left (-\frac1z\right )-M(0)+p^{-1}}
{M\left (-\frac1z\right )-M(0)+\overline{p^{-1}}}&=
\frac{-\int_\bbR\frac{d\mu(\lambda)}{\lambda(1+\lambda z)}+p^{-1}}
{-\int_\bbR\frac{d\mu(\lambda)}{\lambda(1+\lambda z)}+
\overline{p^{-1}}}
=\frac{\overline p}{p}
\cdot \frac{1-p\int_\bbR\frac{d\mu(\lambda)}{\lambda(1+\lambda z)}}
{1-\overline{p}\int_\bbR\frac{d\mu(\lambda)}{\lambda(1+\lambda z)}},
\end{align*}
one concludes that
$$
S_{\fB}\left (-\frac1z\right )=
-\frac{1-\varkappa}{1-\overline{\varkappa}}\cdot
\frac{\overline p}{p}
\cdot\frac{1-p\int_\bbR\frac{d\mu(\lambda)}{\lambda(1+\lambda z)}}
{1-\overline{p}\int_\bbR\frac{d\mu(\lambda)}{\lambda(1+\lambda z)}}, \quad z\in \bbC_+.
$$
Thus, taking into account \eqref{charb}, one obtains
\begin{equation}\label{ams1}
S_{\fB}\left (-\frac1z\right )
=
-\frac{1-\varkappa}{1-\overline{\varkappa}}\cdot \frac{\overline p}{p}
 \cdot S_{-(\widehat B)^{-1}}(z),
 \quad z\in \bbC_+.
\end{equation}
Using \eqref{pp} we get
\begin{align}
-\frac{1-\varkappa}{1-\overline{\varkappa}}\cdot \frac{\overline p}{p}&=-\frac{1-\varkappa}{1-\overline{\varkappa}}\cdot
\frac{M(0)+
i\frac{\varkappa+1}{\varkappa-1}}{M(0)+
i\frac{\overline{\varkappa}+1}{\overline{\varkappa}-1}}
=\frac{M(0)-i-\varkappa (M(0)+i)}{M(0)+i-\overline{\varkappa}(M(0)-i)}\label{ams2}
\\&=\frac{s_{(\dot \cB, \cB)}(0)-\varkappa}{\overline{\varkappa} s_{(\dot \cB, \cB)}(0+i0)-1}=S_{\fB}(0+i0).\nonumber
\end{align}
Combining \eqref{ams1} and \eqref{ams2} shows that 
$$ S_{-(\widehat \cB)^{-1}}\left (-\frac1z\right )=\left (S_{\fB} (0+i0)\right )^{-1}S_{\fB} (z ),
 \quad z\in \bbC_+,
$$
which proves the claim provided that  $\Ker (\cB)=\{0\}$.

To relax the requirement that $\Ker (\cB)=\{0\}$, suppose that $\cB'$ is a self-adjoint extension of $\dot \cB$ such that  $\Ker (\cB')=\{0\}$. Such an extension is alway available since $0\in \widehat \rho(\dot \cB)$.
However, 
$$
\frac{S_{\fB}(z)}{
S_{\fB} (0+i0)}=\frac{S_{\fB'}(z)}{
S_{\fB'} (0+i0)},  \quad z\in \bbC_+,
$$
where $\fB'=(\dot \cB, \widehat \cB, \cB')$ and hence \eqref{kep} holds regardless of whether 
$\Ker (\cB)=\{0\}$ or not.

\end{proof}

\section{Proof of Theorem \ref{main}}\label{Secproof}
  Now we are ready to establish the invariance principle 
 also for the characteristic function of a triple
  for general linear-fractional automorphisms of the upper half-plane.

\begin{proof}  Any linear-fractional automorphism 
 $$\bbC_+\ni z\mapsto f(z)=\frac{az+b}{cz+d},\quad a,b, c, d\in \bbR , $$
$ ad-bc>0$, $c>0$ can be represented as the composition 
$$f=h\circ \iota\circ g,
$$
where $h$ is a  linear automorphism of $\bbC_+$,
$$\iota(z)=-\frac1z
\quad \text{and}\quad 
g(z)=z-f^{-1}(\infty),
\quad z\in \bbC_+.
$$

If $\omega=f^{-1}(\infty) $
belongs to the core of the spectrum of  $\dot A$, the point  $0$ belongs to the core of the spectrum of the symmetric operator $g(\dot A)=\dot A-\omega I$. Therefore, in the model representation of the triple $g(\fA)$ the hypotheses of Theorem \ref{takitak} are satisfied. In particular, the triple $\iota\circ g(\fA)$ 
is regular, so is $f(\fA)$, since  $f=h\circ \iota\circ g$ and $h$ is a linear isomorphism.
By Theorem \ref{coin}, 
$$
\widehat S_{f(\fA)}\circ f=\widehat S_{\iota\circ g (\fA)}\circ (\iota\circ g).
$$
Applying Theorem \ref{takitak}, we get
$$
\widehat S_{\iota\circ g (\fA)}\circ (\iota\circ g)=\widehat S_{ g (\fA)}\circ g=\widehat S_{ \fA},
$$
where we have used Theorem \ref{coin} one more time  in the last step, thus proving \eqref{inveqq}.

If $\omega=f^{-1}(\infty) $
is a quasi-regular point of  $\dot A$, both  $f(\widehat A)$ and $(\iota\circ g) (\widehat A$) are bounded dissipative operators. By the invariance principle in the bounded  case  (see \eqref{law}), we have
$$
 S_{f(\widehat A)}\circ f=S_{\iota\circ g (\widehat A)}\circ (\iota\circ g).
$$
Applying Lemma \ref{tak} we get
$$
S_{\iota\circ g (\widehat A)}\circ (\iota\circ g)=\frac{1}{\widehat S_{g(\fA)}(0+i0)}\cdot 
{\widehat S_{g(\fA)}}\circ g.
$$
By Theorem \ref{coin},
$$
\frac{1}{\widehat S_{g(\fA)}(0+i0)}\cdot 
{\widehat S_{g(\fA)}}\circ g=\frac{1}{\widehat S_{\fA}(\omega+i0)}\cdot \widehat S_{\fA}=\frac{1}{S_{\fA}(\omega+i0)}\cdot S_{\fA}.
$$
Therefore,
$$
S_{f(\widehat A)}\circ f=\frac{1}{S_{\fA}(\omega+i0)}\cdot S_{\fA},
$$
completing the proof  \eqref{inveq2}.
 \end{proof}

\section{Applications to the Krein-von Neumann extensions theory}
Recall that if $ \dot A$ is a densely defined (closed) nonnegative operator, then the set of all nonnegative self-adjoint extensions of $ \dot A$ has the minimal element $[\dot A]_K$ , the Krein-von Neumann extension (different authors refer to the minimal extension  by using different names, see, e.g., \cite{AlSi,AnNi,AT,Bir} , and the maximal one  $[\dot A]_F$ , the Friedrichs extension. This means, in particular, that for any nonnegative self-adjoint extension $A$ of 
 $\dot A$ the following operator inequality holds \cite{Krein0}:
 $$
([\dot A]_F+\lambda I)^{-1} \le (A+\lambda I)^{-1} \le([\dot A]_K +\lambda I)^{-1},\quad  \text{for all}\quad \lambda >0.
$$
 
\begin{hypothesis}\label{scal} Suppose that $\nu\in (-1,1)$. 
Assume that   $A(\nu)$ is  the self-adjoint multiplication operator by independent variable 
in the Hilbert space 
$$\cH_+=L^2((0, \infty); \lambda^\nu d\lambda)$$
  and $\dot A(\nu)$ is its restriction on
$$
\Dom(\dot A(\nu))=\left \{f\in \Dom(A(\nu))\, \big |\, \int_0^\infty f(\lambda)\lambda^\nu d\lambda =0\right \}.
$$

Analogously, suppose that   $B(\nu)$ is  the self-adjoint multiplication operator by independent variable 
in the Hilbert space 
$$\cH_-=L^2((-\infty, 0); |\lambda|^\nu d\lambda)$$ and $\dot B(\nu)$ is its restriction on
$$
\Dom(\dot B(\nu))=\left \{f\in \Dom(B(\nu))\, \big |\, \int_{-\infty}^0 f(\lambda)|\lambda|^\nu d\lambda =0\right \}.
$$

\end{hypothesis}

\begin{lemma}\label{pervaia} Assume Hypothesis \ref{scal}. Then  the Weyl-Titchmarsh function $M_\nu(z)$ associated with the pair
$(\dot A(\nu), A(\nu))$  admits the representation
\begin{equation}\label{repM}
M_\nu(z)=
\begin{cases}
\left (i-\cot \frac\pi2 \nu\right )\left (\frac zi\right)^\nu +\cot \frac\pi2 \nu, &\nu\ne 0\\
\frac2\pi \log\left (-\frac1z\right ),&\nu=0
\end{cases},
\quad z\in \bbC_+.
\end{equation}

Analogously, the Weyl-Titchmarsh function $N_\nu(z)$ associated with the pair
$(\dot B(\nu), B(\nu))$  admits the representation
\begin{equation}\label{repN}
N_\nu(z)=
\begin{cases}
\left (i+\cot \frac\pi2 \nu\right )\left (\frac zi\right)^\nu -\cot \frac\pi2 \nu, &\nu\ne 0\\
\frac2\pi \log z,&\nu=0
\end{cases},
\quad z\in \bbC_+.
\end{equation}
 Here $\log z$ denotes the principal brach of the logarithmic function with the cut on the negative semi-axis
 and
 $$
 \left (\frac zi\right)^\nu=\exp \left [\nu\left ( \log z-i\frac\pi2\right )\right ],\quad z\in \bbC_+.
 $$

\end{lemma}
\begin{proof} Suppose that $\nu\ne 0$.  We have
$$
M_\nu(z)=\frac{1}{\|g_+\|^2}\left ((zA(\nu)+I)(A(\nu)-zI)^{-1}g_+,g_+ \right ), \quad z\in \bbC_+,
$$
where $g_+$ is a deficiency element from $\Ker ((\dot A(\nu))^*-iI)$.
One can choose (see, e.g., \cite{MT-S})
$$
g_+(z)=\frac{1}{\lambda-z}, \quad z\in \bbC_+, \quad \lambda\in (0, \infty).
$$
We have
$$
\left ((zA(\nu)+I)(A(\nu)-zI)^{-1}g_+,g_+ \right )=\int_0^\infty \frac{(z\lambda+1)\lambda^\nu}{(\lambda-z)(1+\lambda^2)}d\lambda
$$
and 
$$
\|g_+\|^2=\int_0^\infty \frac{\lambda^\nu}{1+\lambda^2}d\lambda.
$$

Let $\Gamma_\varepsilon $, $\varepsilon >0$, denote the anti-clockwise oriented contour in the complex plane 
\begin{align*}\Gamma_\varepsilon=&\left \{z\in\bbC\,|\,\text{dist}(z, [0,\infty),|z|\le \varepsilon^{-1} \right \}
\\& \cup\left \{z\in\bbC\,|\,|z|=\varepsilon^{-1}, \arg (z)\in
[\arcsin (\varepsilon^2), 2\pi -\arcsin( \varepsilon^2]\right \},
\end{align*}
which consists of ``an infinitely distant circle and an indentation round the cut along the 
positive real axis"
 \cite[\textsection \,129, Fig. 46]{LL} (as $\varepsilon \to 0)$.

Given $z\in \bbC_+$ and   $\varepsilon<\min(|z|,1)$,
by the Residue theorem we get
$$
\varointctrclockwise\limits_{\Gamma_\varepsilon} \frac{(z\lambda+1)\lambda^\nu}{(\lambda-z)(1+\lambda^2)}d\lambda=2\pi i \sum_{\zeta\in\{z, i, -i\}} \text{Res}( F,\zeta),
$$
where
$$
F(\lambda)=\frac{(z\lambda+1)\lambda^\nu}{(\lambda-z)(1+\lambda^2)}.
$$
Going to the limit as $\varepsilon\to 0$ and taking into account that 
$$
(\lambda-i0)^\nu=e^{i2\pi\nu}\lambda^\nu,\quad  \lambda>0,
$$
we arrive at the representation
\begin{align*}
(1-e^{i2\pi\nu })\int_0^\infty \frac{(z\lambda+1)\lambda^\nu}{(\lambda-z)(1+\lambda^2)}d\lambda&=
\lim_{\varepsilon\downarrow 0}\varointctrclockwise\limits_{\Gamma_\varepsilon} \frac{(z\lambda+1)\lambda^\nu}{(\lambda-z)(1+\lambda^2)}d\lambda
\\&=2\pi i \sum_{\zeta\in\{z, i, -i\}} \text{Res}( F,\zeta).
\end{align*}
Analogously, 
$$
(1-e^{i2\pi\nu })\int_0^\infty \frac{\lambda^\nu}{1+\lambda^2}d\lambda=2\pi i \sum_{\zeta\in\{ i, -i\}} \text{Res}( G,\zeta),
$$
where
$$
G(\lambda)=\frac{\lambda^\nu}{1+\lambda^2}.
$$

A direct computation shows that 
\begin{align*}M_\nu(z)&=\frac{\sum_{\zeta\in\{z, i, -i\}} \text{Res}( F,\zeta)
}{\sum_{\zeta\in\{ i, -i\}} \text{Res}( G,\zeta)}=\frac{z^\nu-\frac{i^\nu+(-i)^{\nu}}{2}}{\frac{i^\nu-(-i)^\nu}{2i}}
\\&=\frac{2i}{1-e^{i\pi \nu}}\left (\frac zi\right )^\nu-i\frac{1+e^{i\pi \nu}}{1-e^{i\pi \nu}}
\\&=\left (i-\cot \frac\pi2 \nu\right )\left (\frac zi\right)^\nu +\cot \frac\pi2 \nu, \quad z\in \bbC_+.
\end{align*}
The case of $\nu=0$ can be justified by taking the limit
\begin{equation}
M_0(z)=\lim_{\nu \to 0}M_\nu^+(z)=\frac2\pi \log\left (-\frac1z\right ), \quad z\in \bbC_+,
\end{equation}
which competes the proof of \eqref{repM}.

The proof of \eqref{repN} is analogous.
\end{proof}

\begin{lemma}\label{FK}
Assume Hypothesis \ref{scal}.

Then, 
if $\nu\in [0,1)$, then $A(\nu)$ is the Friedrichs extension of $\dot A(\nu)$.
If $\nu\in (-1,0]$, then $A(\nu)$ is the Krein-von Neumann extension of $\dot A(\nu)$.

In particular,   the Friedrichs and Krein-von Neumann extensions of $\dot A(0)$ coincide.
\end{lemma}
\begin{proof}

By Lemma \ref{pervaia}, the Weyl-Titchmarsh function associated with the pair  $(\dot A(\nu), A(\nu))$ can be evaluated as
$$M_{(\dot A(\nu), A(\nu))}(z)=M_\nu(z), \quad  z\in \bbC_+.
$$

 To complete the proof it remains to  observe that 
 $$
 \lim_{\lambda\downarrow -\infty}M_\nu(\lambda)=-\infty, \quad 0\le \nu<1,
 $$
 and 
\begin{equation}\label{krkr}
 \lim_{\lambda\uparrow 0}M_\nu(\lambda)=\infty, \quad  -1<\nu\le 0,
 \end{equation}
and then  apply  \cite[Theorem 4.4]{GT},  a result that   characterizes the Weyl-Titchmarsh function threshold behavior
for the Friedrichs and Krein-von Neumann extensions, respectively.
\end{proof}
\begin{remark}\label{unitara}
Notice that  the Cayley transforms of $M_\nu(z)$ and $M_{-\nu}(z)$ coincide up to a constant unimodular factor,
that is,

$$
\frac{M_{\nu}(z)-i}{M_{\nu}(z)+i}=
\frac{\left (i-\cot \frac\pi2 \nu\right )\left (\frac zi\right)^\nu +\cot \frac\pi2 \nu-i}{
\left (i-\cot \frac\pi2 \nu\right )\left (\frac zi\right)^\nu +\cot \frac\pi2 \nu+i}
=\frac{\left (\frac zi\right)^\nu -1}{\left (\frac zi\right)^\nu -e^{i\pi \nu}},
$$
\begin{align*}
\frac{M_{-\nu}(z)-i}{M_{-\nu}(z)+i}&=
\frac{\left (i+\cot \frac\pi2 \nu\right )\left (\frac zi\right)^{-\nu} -\cot \frac\pi2 \nu-i}{
\left (i+\cot \frac\pi2 \nu\right )\left (\frac zi\right)^{-\nu} -\cot \frac\pi2 \nu+i}
=\frac{\left (\frac zi\right)^{-\nu} -1}{\left (\frac zi\right)^{-\nu} -e^{-i\pi \nu}}
\\&=e^{i\pi \nu}\frac{\left (\frac zi\right)^\nu -1}{\left (\frac zi\right)^\nu -e^{i\pi \nu}},\quad z\in \bbC_+.
\end{align*}
Therefore, by \eqref{blog}, the Liv\v{s}ic functions $s_\nu(z)$ and $s_{-\nu}(z)$
 associated with the pairs $(\dot A(\nu), A(\nu))$ and $(\dot A(-\nu), A(-\nu))$, respectfully, are related as
$$
s_\nu(z)=e^{i\pi \nu}s_{-\nu}(z),\quad z\in \bbC_+.
$$
Since the knowledge (up to a unimodular factor)  of the Liv\v{s}ic function $s_{(\dot A, A)}(z)$  of a pair $(\dot A, A)$,
 where $\dot A$ is a prime symmetric operator and $A$ its self-adjoint extension, determines
the symmetric operator $\dot A$ up to a unitary equivalence (see \cite{AkG,L46,MTH,MTBook}), 
and,  moreover,  $\dot A(\nu)$ and $\dot A (-\nu)$ are prime symmetric operators, we conclude that 
 $\dot A(\nu)$ and $\dot A (-\nu)$ are unitarily equivalent.
\end{remark}

The following  theorem addresses   ``intertwining''  properties of 
the Friedrichs and Krein-von Neumann extensions of an operator with respect to  the inverse operation.

\begin{theorem}\label{last} Assume Hypothesis \ref{scal} and set $\dot A=\dot A(\nu)$ and $A= A(\nu)$.
Then
\begin{equation}\label{FKinv}
\left ([\dot A]_F\right )^{-1}=\left [(\dot A)^{-1}\right ]_K
\end{equation}
and
\begin{equation}\label{KFinv}
\left ([\dot A]_K\right )^{-1}=\left [(\dot A)^{-1}\right ]_F.
\end{equation}
\end{theorem}
\begin{remark}
Notice that the operators   $\dot A (\nu)$,  $A (\nu)$, etc.,
referred to in Hypothesis  \ref{scal}
are essentially coincide with the model multiplication operators in the  weighted Hilbert space $L^2((0,\infty);d\mu)$ given by \eqref{nacha1}-\eqref{nacha3}   (after an appropriate renormalization of the weight $d\mu(\lambda)=\lambda^\nu d\lambda$), so that Theorem \ref{takitak} applies.
For instance, the inverse of $A=\dot A (\nu)$ is well defined as a prime symmetric operator with deficiency indices $(1,1)$.

\end{remark}

\begin{proof} Comparing \eqref{repM} and \eqref{repN} one observes that for $\nu \in (-1,1)$ we have 
\begin{equation}\label{MN}
M_\nu\left (-\frac1z\right )=N_{-\nu}(z), \quad z\in \bbC_+.
\end{equation}
By Theorem \ref{takitak} (vi),  the left hand side of \eqref{MN}  is the Weyl-Titchmarsh function associated with the pair 
$(-(\dot A(\nu))^{-1}, -(A(\nu))^{-1})$ and therefore
$$
(-(\dot A(\nu))^{-1}, -(A(\nu))^{-1})\cong((\dot B(-\nu))^{-1}, (B(-\nu))^{-1}).
$$
Here the symbol $\cong$ denotes the mutual unitary equivalence of the corresponding pairs.
From the definition of the operators $A(\nu)$ and $B(\nu)$ it follows that 
$$
((\dot B(-\nu))^{-1}, (B(-\nu))^{-1})\cong 
(-\dot A(-\nu), -A(-\nu))
$$
and therefore
$$
((\dot A(\nu))^{-1}, (A(\nu))^{-1})\cong(\dot A(-\nu), A(-\nu)).
$$

Suppose that $\nu\ge 0$. By Lemma \ref{FK},
$$
A(\nu)=[\dot A(\nu)]_F.
$$
So that $$
((\dot A(\nu))^{-1}, (A(\nu))^{-1})=((\dot A(\nu))^{-1}, ([\dot A(\nu)]_F)^{-1})
$$
and hence
\begin{equation}\label{muteq}
((\dot A(\nu))^{-1}, ([\dot A(\nu)]_F)^{-1})\cong(\dot A(-\nu), A(-\nu)).
\end{equation}
Again, by Lemma \ref{FK},  the operator $A(-\nu)$ is the Krein-von Neumann extension of $ \dot A(-\nu)$. From the mutual unitary equivalence of the pairs \eqref{muteq} it follows that the second operator from the left pair  is the Krein-von Neumann extension of 
the first one. That is,
$$
([\dot A(\nu)]_F)^{-1}=[(\dot A(\nu))^{-1}]_K
$$
or,  equivalently,
$$
([\dot A]_F)^{-1}=[(\dot A)^{-1}]_K,
$$
which proves 
\eqref{FKinv} (for $\nu\in [0, 1)$).

One also has that
$$
((\dot A(-\nu))^{-1}, (A(-\nu))^{-1})\cong(\dot A(\nu), A(\nu)).
$$
The  same reasoning shows that 
$$
[(\dot A(-\nu))^{-1}]_F=([\dot A(-\nu)]_K)^{-1}.
$$
However, by Remark \ref{unitara},  the operators $\dot A(-\nu)$ and $\dot A(\nu)$ are unitarily equivalent which implies   that
$$
[(\dot A(\nu))^{-1}]_F=([\dot A(\nu)]_K)^{-1}
$$
or, equivalently,
$$
[(\dot A)^{-1}]_F=([\dot A]_K)^{-1},
$$
and \eqref{KFinv} follows (for $\nu\in [0, 1)$).

The proof for $\nu\in (-1,0)$ is analogous.

\end{proof}
\begin{remark}
 In view of Remark \ref{unitara}, from \eqref{muteq} it also follows that the symmetric operators $\dot A$ and $(\dot A)^{-1}$ 
 referred to in Theorem \ref{last} are unitarily equivalent.
\end{remark}
\begin{remark} 
In connection with \eqref{KFinv}
it is worth mentioning  that  in a more general context the  (Krein-)von Neumann extension  $S_N$ of a positive subspace
$S$   has been defined in \cite{Cod} as
$$
S_N=([S^{-1}]_F)^{-1},
$$
where $[S^{-1}]_F$ is the Friedrichs extensions of the relation $S^{-1}$  (see  also \cite{AnNi} and \cite{AT}).

\end{remark}

{\bf Acknowledgement}. We are very grateful to S. Belyi for his help in the submission process  
of  this article.

\end{document}